\documentclass[12pt]{amsart}
\usepackage{amssymb}

\allowdisplaybreaks

\newtheorem{theorem}[equation]{Theorem}
\newtheorem{lemma}[equation]{Lemma}
\newtheorem{proposition}[equation]{Proposition}
\newtheorem{corollary}[equation]{Corollary}

\theoremstyle{definition}

\newtheorem*{acknowledgements}{Acknowledgements}
\newtheorem{example}[equation]{Example}
\newtheorem{remark}[equation]{Remark}

\theoremstyle{remark}

\numberwithin{equation}{section}

\newcommand{\FF}{\mathbb{F}}
\newcommand{\ZZ}{\mathbb{Z}}

\newcommand{\NN}{\mathbb{N}}
\newcommand{\TT}{\mathbb{T}}
\newcommand{\GG}{\mathbb{G}}
\newcommand{\LL}{\mathbb{L}}
\newcommand{\CC}{\mathbb{C}}

\newcommand{\bA}{\mathbf{A}}
\newcommand{\bff}{\mathbf{f}}

\newcommand{\bS}{\mathbf{S}}

\newcommand{\cB}{\mathcal{B}}
\newcommand{\cL}{\mathcal{L}}

\newcommand{\cP}{\mathcal{P}}

\newcommand{\rP}{\mathrm{P}}

\newcommand{\La}{\Lambda}

\DeclareMathOperator{\GL}{GL}
\DeclareMathOperator{\Mat}{Mat}

\DeclareMathOperator{\Res}{Res}

\DeclareMathOperator{\Sol}{Sol}

\newcommand{\oK}{\overline{K}}

\newcommand{\tpi}{\widetilde{\pi}}

\newcommand{\Ga}{\GG_{\mathrm{a}}}

\newcommand{\power}[2]{{#1 [[ #2 ]]}}
\newcommand{\laurent}[2]{{#1 (( #2 ))}}

\begin{document}

\title[Identities for Anderson generating functions]{Identities for
Anderson generating functions \\ for Drinfeld modules}

\author{Ahmad El-Guindy}
\address{Current address: Science Program, Texas A{\&}M University in Qatar,
Doha, Qatar}
\address{Permanent address: Department of Mathematics, Faculty of Science,
Cairo University, Giza, Egypt 12613}
\email{a.elguindy@gmail.com}

\author{Matthew A. Papanikolas}
\address{Department of Mathematics, Texas A{\&}M University, College Station,
TX 77843, USA}
\email{map@math.tamu.edu}

\thanks{Research of the second author was partially supported by NSF Grants DMS-0903838 and
DMS-1200577}


\date{16 July 2013}

\begin{abstract}
Anderson generating functions are generating series for division values of points on Drinfeld modules, and they serve as important tools for capturing periods, quasi-periods, and logarithms.  They have been fundamental in recent work on special values of positive characteristic $L$-series and in transcendence and algebraic independence problems.  In the present paper we investigate techniques for expressing Anderson generating functions in terms of the defining polynomial of the Drinfeld module and determine new formulas for periods and quasi-periods.
\end{abstract}


\maketitle

\section{Introduction}

The definition of Anderson generating functions dates back to Anderson's original paper on $t$-motives \cite{Anderson86}, where he used these generating series to prove fundamental results on the uniformization of abelian $t$-modules.  He expanded on them in \cite{Anderson93}, where he used them to provide links between $v$-adic and de Rham realizations of $t$-modules and $t$-motives.

These functions are also central in many other developments in function field arithmetic.  After fixing a $(q-1)$-st root of $-\theta$, the Anderson generating function,
\begin{equation}
  \omega_C(t) = \sum_{m=0}^\infty \exp_C \biggl( \frac{\tpi}{\theta^{m+1}} \biggr) t^m
  = (-\theta)^{1/(q-1)} \prod_{i = 0}^\infty \biggl( 1- \frac{t}{\theta^{q^i}} \biggr)^{-1},
\end{equation}
for the Carlitz module, defined by Anderson and Thakur~\cite{AndThak90}, allows for explicit descriptions of the Carlitz period $\tpi$ and polylogarithms on tensor powers of the Carlitz module.  Indeed, we can calculate $\tpi$ in terms of the residue of $\omega_C(t)$ at $t=\theta$,
\begin{equation} \label{CarlitzPerFormula}
  \tpi = -\Res_{t=\theta} \omega_C(t) = \theta(-\theta)^{1/(q-1)} \prod_{i=1}^\infty
  \Bigl( 1 - \theta^{1-q^i} \Bigr)^{-1},
\end{equation}
which essentially dates back to Carlitz \cite[Thm.~5.1]{Carlitz35}.  More recently $\omega_C(t)$ has proved to be a crucial element in the understanding of vector valued Drinfeld modular forms and also the evaluation of positive characteristic $L$-series of Pellarin; see Pellarin~\cite{Pellarin11}, \cite{Pellarin12}, \cite{Pellarin13}, and Perkins~\cite{Perkins12}.

Anderson generating functions have also been central to transcendence theory, as they are solutions of certain Frobenius difference equations (see also Theorem~\ref{MainThm} in the present paper), in particular to algebraic independence results on periods, quasi-periods, and logarithms.  See for example, work of Anderson, Brownawell, and the second author~\cite{ABP04} and of the second author~\cite{P08}, for connections with special $\Gamma$-values and Carlitz logarithms.  As observed in papers of Gekeler~\cite{Gekeler89a}, \cite{Gekeler89b}, using the theory of biderivations of Yu~\cite{Yu90}, Anderson generating functions for general Drinfeld modules can also be used to evaluate periods and quasi-periods, and this point of view has been useful in transcendence theory for higher rank Drinfeld modules; see work of Chang and the second author \cite{CP11}, \cite{CP12}, which has its beginnings in work of Pellarin~\cite{Pellarin08}.

The goal of the present paper, culminating in our main results in Theorem~\ref{MainThm}, is to find explicit expressions for Anderson generating functions using only the defining polynomial of the Drinfeld module.  If
\[
  \phi_t = \theta + A_1 \tau + \cdots + A_r\tau^r
\]
defines a Drinfeld module $\phi$ over $\CC_\infty$, then for fixed $\xi \in \CC_\infty$ sufficiently small (see Theorem~\ref{MainThm} for a precise bound), we define an explicit series in terms of the parameters $A_1, \dots, A_r$,
\begin{equation}
    \cL_\phi(\xi;t)
  = \xi + \sum_{n=1}^\infty \left( \sum_{\bS \in P_r(n)}\prod_{i=1}^r \prod_{j \in S_i}\frac{A_i^{q^j}}{t-\theta^{q^{i+j}}} \right) \xi^{q^n},
\end{equation}
where the interior sum is finite and taken over certain partitions $(S_1, \dots, S_r)$ of the set $\{0, 1, \dots, n-1\}$ called shadowed partitions (see \S\ref{SP} or \cite[\S 2]{EP}).
The series defining $\cL_\phi(\xi;t)$ converges in the Tate algebra $\TT$ of functions on the closed unit disk of $\CC_\infty$ and is regular at $t=\theta$.

The series $\cL_\phi(\xi;t)$ provides a deformation in the $t$ parameter of the Drinfeld logarithm $\log_\phi(\xi)$ in that formulas of the authors from \cite[\S 3]{EP} imply
\begin{equation}
  \cL_\phi(\xi;\theta) = \log_\phi(\xi).
\end{equation}
As such it generalizes functions of this type for the Carlitz module, which were studied by the second author in \cite[\S 6]{P08}.
It is also directly related to Anderson generating functions:
\begin{equation}
  \cL_\phi(\xi;t) = -(t-\theta) f_\phi(\log_\phi(\xi);t),
\end{equation}
where $f_\phi(\log_\phi(\xi);t)$ is the Anderson generating function associated to $\log_\phi(\xi)$.  Using techniques from \cite{CP11}, \cite{CP12}, \cite{Pellarin08}, we can also express periods and quasi-periods of $\phi$ in terms of values of $\cL_\phi(\xi;t)$ and its Frobenius twists (see \S\ref{ApplQP}).  For more details, Theorem~\ref{MainThm} determines the fundamental properties of $\cL_\phi(\xi;t)$.

The outline of the paper is as follows.  After reviewing definitions about Drinfeld modules and Frobenius difference equations in \S\ref{prelim}, we define and discuss Anderson generating functions in \S\ref{AGF}.  We then investigate the case of the Carlitz module in \S\ref{CarlitzSec}, which centers on the power series $\omega_C(t)$.  In \S\ref{SP} we review facts about shadowed partitions and their applications to Drinfeld exponentials and logarithms, and then in \S\ref{ExplAGF} we prove Theorem~\ref{MainThm} about $\cL_\phi(\xi;t)$ and its connections with Anderson generating functions.  Finally, in \S\ref{ApplQP} we derive new formulas for quasi-periods using Theorem~\ref{MainThm} and discuss their implications to Anderson's Legendre relation in rank~$2$.

\begin{acknowledgements}
We thank the referee for carefully reading the paper and for several helpful suggestions.
\end{acknowledgements}

\section{Function fields and Drinfeld modules} \label{prelim}

Let $q$ be a fixed power of a prime $p$, and let $K = \FF_q(\theta)$
be the rational function field in the variable $\theta$ over the
finite field with $q$ elements.  Let $K_\infty =
\laurent{\FF_q}{1/\theta}$ be the completion of $K$ at $\infty$, with absolute value $\lvert\, \cdot\, \rvert$ chosen so that $|\theta| = q$.  We let $v_\infty$ be the valuation at $\infty$ with $v_\infty(\theta)=-1$, and let $\deg := -v_\infty$.
Let $\CC_\infty$ be the completion of an algebraic closure of
$K_\infty$.

Let $\tau : \CC_\infty \to \CC_\infty$ denote the $q$-th power Frobenius map, and take $\CC_\infty[\tau]$ to be the subalgebra of $\FF_q$-linear endomorphisms of the additive group of $\CC_\infty$.  In this way $\CC_\infty[\tau]$ forms a ring of twisted polynomials in $\tau$, subject to the relation $\tau c = c^q \tau$ for $c \in \CC_\infty$.  Then $\CC_\infty$ is a left $\CC_\infty[\tau]$-module, where if $\psi = a_0 + a_1 \tau + \dots + a_m\tau^m$, then
\[
  \psi(x) = a_0 x + a_1 x^q + \dots + a_m x^{q^m}, \quad x \in \CC_\infty.
\]

Now let $\FF_q[t]$ be the polynomial ring in the variable $t$ independent from $\theta$.  A Drinfeld module of rank $r$ over $\CC_\infty$ is an $\FF_q$-algebra homomorphism
\[
  \phi : \FF_q[t] \to \CC_\infty[\tau],
\]
defined so that
\begin{equation} \label{phidef}
  \phi_t = \theta + A_1 \tau + \dots + A_r \tau^r, \quad A_r \neq 0.
\end{equation}
We will assume throughout that $r \geq 1$.
The map $\phi$ induces an $\FF_q[t]$-module action on~$\CC_\infty$,
\[
  a \cdot x = \phi_a(x), \quad x \in \CC_\infty.
\]
The exponential of $\phi$ is defined to be the $\FF_q$-linear power series in $z$,
\begin{equation} \label{expphidef}
  \exp_\phi(z) = \sum_{n=0}^\infty \alpha_n z^{q^n}, \quad \alpha_0 = 1,\ \alpha_n \in \CC_\infty,
\end{equation}
that satisfies $\exp_\phi( a(\theta) z) = \phi_a(\exp_\phi(z))$, $a \in \FF_q[t]$.
As a function $\exp_\phi : \CC_\infty \to \CC_\infty$, the exponential is entire, $\FF_q$-linear, and surjective.  The logarithm of $\phi$ is defined to be the formal inverse of $\exp_\phi(z)$,
\begin{equation} \label{logphidef}
  \log_\phi(z) = \sum_{n=0}^\infty \beta_n z^{q^n}, \quad \beta_0 = 1,\ \beta_n \in \CC_\infty,
\end{equation}
which has a finite radius of convergence (equal to the absolute value of the smallest period of $\phi$~\cite[Prop.~4.14.2]{GossBook}; or e.g., see \cite[Cor.~4.2]{EP}, where the radius of convergence is worked out explicitly in rank~$2$).  For more details on Drinfeld modules the reader is directed to \cite{GossBook}, \cite{ThakurBook}.

Define the Tate algebra
\[
  \TT = \left\{ \sum_{i=0}^\infty c_i t^i \in \power{\CC_\infty}{t} :
\lvert c_i \rvert \to 0 \right\},
\]
to be the ring of power series that converge on the closed unit disk of
$\CC_\infty$.  We define a norm $\lVert \, \cdot\, \rVert$ on $\TT$ by setting
for $f = \sum c_i t^i$,
\[
  \lVert f \rVert = \sup_i |c_i| = \max_i |c_i|,
\]
and in this way $\lVert\, \cdot\,\rVert$ is a complete non-archimedean
absolute value on $\TT$ satisfying
\[
  \lVert cf \rVert = |c|\cdot  \lVert f \rVert, \quad \lVert fg \rVert
= \lVert f \rVert \cdot \lVert g \rVert, \quad c \in \CC_\infty,\ f, g \in \TT.
\]
We further let $\LL$ be the fraction field of $\TT$.  See \cite[\S 2.2]{FresnelvdPut} or \cite[\S 2.2]{P08} for more details.

For $\ell \in \ZZ$ define the $\ell$-th Frobenius twist automorphism on the
Laurent series field $\laurent{\CC_\infty}{t}$ by
\[
  f \mapsto f^{(\ell)} : \sum c_i t^i \mapsto \sum c_i^{q^\ell} t^i.
\]
For a matrix $M \in \Mat_{r \times s}\bigl( \laurent{\CC_\infty}{t} \bigr)$, we take $M^{(\ell)}$ to be the matrix obtained by taking the $\ell$-th twist of the entries of $M$.
The rings $\CC_\infty[t]$, $\CC_\infty(t)$, $\power{\CC_\infty}{t}$, $\TT$, and $\LL$ are invariant under Frobenius twisting.

We define the ring $\CC_\infty(t)[\tau]$ to be the non-commutative ring of polynomials in $\tau$ with coefficients in $\CC_\infty(t)$, subject to the conditions
\[
  \tau g = g^{(1)} \tau, \quad g \in \CC_\infty(t),
\]
and in this way $\laurent{\CC_\infty}{t}$ forms a left $\CC_\infty(t)[\tau]$-module, where if
\[
  \Delta = g_m\tau^m + \cdots + g_1 \tau + g_0, \quad g_i \in \CC_\infty(t),
\]
then
\[
  \Delta(f) = g_m f^{(m)} + \dots + g_1 f^{(1)} +  g_0 f, \quad f \in \laurent{\CC_\infty}{t}.
\]
For $\Delta \in \CC_\infty(t)[\tau]$, set
\[
  \Sol(\Delta) = \{ f \in \LL : \Delta(f) = 0 \}.
\]
We know that $\Sol(\Delta)$ forms an $\FF_q(t)$-vector space of dimension at most $m = \deg_\tau \Delta$ by the following.  If we let $\Upsilon \in \GL_m\bigl(\CC_\infty(t)\bigr)$ and set
\[
  \Sol(\Upsilon) = \{ \bff \in \Mat_{m\times 1}(\LL) : \bff^{(1)} = \Upsilon \bff \},
\]
then $\Sol(\Upsilon)$ is an $\FF_q(t)$-vector space and
\begin{equation} \label{SolPhiDim}
   \dim_{\FF_q(t)} \Sol(\Upsilon) \leq m
\end{equation}
(see \cite[\S 4.1]{P08}, with $(F,K,L) = (\FF_q(t), \CC_\infty(t),\LL)$, or \cite[Ch.~1]{vdPutSinger}).  Now given $\Delta \in \CC_\infty(t)[\tau]$ as above, form the companion matrix
\[
  \Upsilon_\Delta = \begin{pmatrix}
  0 & 1 & \cdots & 0 \\
  & 0 & \ddots & \vdots \\
  & & \ddots & 1 \\
  -\frac{g_0}{g_m} & -\frac{g_1}{g_m} & \cdots & -\frac{g_{m-1}}{g_m}
  \end{pmatrix}.
\]
Then the map
\[
  f \mapsto \begin{pmatrix} f \\ f^{(1)} \\ \vdots \\ f^{(m-1)} \end{pmatrix}
  : \Sol(\Delta) \to \Sol(\Upsilon_\Delta),
\]
is injective and $\FF_q(t)$-linear.  Thus $\dim_{\FF_q(t)} \Sol(\Delta) \leq \dim_{\FF_q(t)} \Sol(\Upsilon_\Delta) \leq m$.

\begin{example} \label{Ex:omega}
Let $\Delta = \tau - (t-\theta)$.  After fixing a $(q-1)$-st root of $-\theta$, consider
\begin{equation} \label{omegaprod}
  \omega_C(t) = (-\theta)^{1/(q-1)} \prod_{i=0}^\infty \left( 1 - \frac{t}{\theta^{q^i}} \right)^{-1},
\end{equation}
which converges in $\TT$, since $\lVert t/\theta^{q^i} \rVert \to 0$ as $i \to \infty$.  It is also important that $\omega_C(t) \in \TT^{\times}$,
as it is non-vanishing on the closed unit disk in $\CC_\infty$, and moreover $1/\omega_C(t)$ converges on all of $\CC_\infty$.  The function
\[
  \Omega_C(t) = 1/\omega_C(t)^{(1)}
\]
was studied extensively in \cite[\S 5]{ABP04} and \cite[\S 3.3]{P08}.  A short calculation yields
\begin{equation} \label{omegaCdiffeq}
  \Delta \bigl( \omega_C(t) \bigr) = \omega_C(t)^{(1)} - (t-\theta)\omega_C(t) = 0.
\end{equation}
Since $\deg_\tau \Delta = 1$, it follows that $\Sol(\Delta) = \FF_q(t) \cdot \omega_C(t)$.
\end{example}

\section{Anderson generating functions} \label{AGF}

We continue with the notation of \S\ref{prelim}, and assume that a Drinfeld module $\phi : \FF_q[t] \to \CC_\infty[\tau]$ has been chosen as in \eqref{phidef}.  For $u \in \CC_\infty$, we define the Anderson generating function,
\begin{equation} \label{AGFdef}
  f_\phi(u; t) = \sum_{m=0}^\infty \exp_\phi \biggl( \frac{u}{\theta^{m+1}} \biggr) t^m \in \power{\CC_\infty}{t}.
\end{equation}
We have the following proposition due to Pellarin, where we recall from \eqref{expphidef} that $\exp_\phi(z) = \sum \alpha_i z^{q^i}$ with $\alpha_0=1$.

\begin{proposition}[{Pellarin~\cite[\S 4.2]{Pellarin08}}] \label{AGFproperties}
For $u \in \CC_\infty$, we have an identity in $\TT$,
\[
  f_\phi(u; t) = \sum_{n = 0}^{\infty} \frac{ \alpha_n u^{q^n}}{\theta^{q^n}-t}.
\]
Furthermore, $f_\phi(u; t)$ extends to a meromorphic function on $\CC_\infty$ with simple poles at $t= \theta^{q^n}$, $n =0$, $1, \ldots$, with residues
\[
  \Res_{t=\theta^{q^n}} f_\phi(u; t) = -\alpha_n u^{q^n}.
\]
In particular, $\Res_{t=\theta} f_\phi(u; t) = -u$.
\end{proposition}

\begin{proof}
Expanding out $f_\phi(u;t)$ in $\power{\CC_\infty}{t}$, we have
\[
  f_\phi(u;t) = \sum_{m=0}^\infty \left( \sum_{n=0}^\infty \frac{\alpha_n u^{q^n}}{\theta^{q^n(m+1)}}\right) t^m.
\]
If we consider
\[
  \left| \frac{\alpha_n u^{q^n}}{\theta^{q^n(m+1)}} \right|
  = \left| \frac{\alpha_n u^{q^n}}{\theta^{q^n}} \right| \cdot
  \frac{1}{\lvert \theta \rvert^{q^n \cdot m}},
\]
the fact that $\exp_\phi(z)$ converges on all of $\CC_\infty$ implies that the first factor on the right is bounded independent of $n$, and thus the coefficients of $t^m$ in $f_\phi(u; t)$ go to $0$ in $\CC_\infty$ as $m \to \infty$.  Thus $f_\phi(u; t) \in \TT$.  Moreover, for fixed $n \geq 0$, we have the identity in $\TT$,
\[
  \frac{\alpha_n u^{q^n}}{\theta^{q^n}} \sum_{m=0}^\infty \left( \frac{t}{\theta^{q^n}} \right)^m
  = \frac{ \alpha_n u^{q^n}}{\theta^{q^n}-t}.
\]
Now again using that $\exp_\phi(z)$ is entire, and so $|\alpha_n|$ decay rapidly, it follows that
\[
  \lim_{n \to \infty} \left\lVert \frac{ \alpha_n u^{q^n}}{\theta^{q^n}-t} \right\rVert
 = 0.
\]
Therefore we have convergence in $\TT$ of
\[
  \sum_{n =0}^\infty \frac{ \alpha_n u^{q^n}}{\theta^{q^n}-t}  = \sum_{n=0}^\infty
  \frac{\alpha_n u^{q^n}}{\theta^{q^n}} \sum_{m=0}^\infty \left( \frac{t}{\theta^{q^n}} \right)^m
  = f_\phi(u; t).
\]
For the remainder of the proposition, we see that for any $N \geq 0$,
\[
  f_\phi(u; t)- \sum_{n =0}^N \frac{ \alpha_n u^{q^n}}{\theta^{q^n}-t} = \sum_{n =N+1}^\infty \frac{ \alpha_n u^{q^n}}{\theta^{q^n}-t},
\]
which converges for $t=z$ with $|z| < |\theta|^{q^{N+1}}$.  This provides the desired meromorphic continuation, and the residue calculations follow similarly.
\end{proof}

\section{The Carlitz module} \label{CarlitzSec}

In this section we consider Anderson generating functions for the Carlitz module, which provides our motivation for Theorem~\ref{MainThm}.  Let $C : \FF_q[t] \to \CC_\infty[\tau]$ be the Drinfeld module defined by
\[
  C_t = \theta + \tau.
\]
As proved originally by Carlitz~\cite{Carlitz35} (see \cite[Ch.~3]{GossBook} or \cite[Ch.~2]{ThakurBook}), we have that
\begin{equation} \label{expClogC}
  \exp_C(z) = \sum_{n=0}^\infty \frac{z^{q^n}}{D_n}, \quad
  \log_C(z) = \sum_{n=0}^\infty \frac{z^{q^n}}{L_n},
\end{equation}
where $D_0=L_0=1$ and for $n \geq 1$ we set
\begin{align*}
  [n] &= \theta^{q^n} - \theta, \\
  D_n &= [n] [n-1]^{q} \cdots [1]^{q^{n-1}},\\
  L_n &= (-1)^n [1][2] \cdots [n].
\end{align*}
The kernel of $\exp_C(z)$ is a rank $1$ lattice over $\FF_q[\theta]$ in $\CC_\infty$, generated by the element $\tpi$ from~\eqref{CarlitzPerFormula}, called the Carlitz period.

For $u \in \CC_\infty$, we have from Proposition~\ref{AGFproperties} that
\[
  f_{C}(u; t) = \sum_{m=0}^\infty \exp_C\left( \frac{u}{\theta^{m+1}} \right) t^m
 = \sum_{n = 0}^\infty \frac{u^{q^n}}{D_n(\theta^{q^n} - t)},
\]
If we let $\xi = \exp_C(u)$, then since $D_{n+1} = D_n^q (\theta^{q^{n+1}}-\theta)$,
\begin{align*}
  f_{C}(u; t)^{(1)} = \sum_{n=0}^\infty \frac{u^{q^{n+1}}}{D_n^q(\theta^{q^{n+1}} - t)}
  &= \sum_{n=0}^\infty \frac{u^{q^n} (\theta^{q^n}-\theta)}{D_n (\theta^{q^n}-t)} \\
  &= \sum_{n=0}^\infty \frac{u^{q^n} (\theta^{q^n}-t)}{D_n (\theta^{q^n}-t)}
  + \sum_{n=0}^\infty \frac{u^{q^n} (t-\theta)}{D_n (\theta^{q^n}-t)}\\
  &= \xi + (t-\theta)f_{C}(u; t).
\end{align*}
Thus, if we take $\Delta = \tau - (t-\theta)$, then
\begin{equation} \label{AGFCarlitzdiffeq}
  \Delta\bigl( f_{C}(u; t) \bigr) = \xi.
\end{equation}

Now if $u=\tpi$, then $\xi=\exp_C(\tpi)=0$, and so
\begin{equation} \label{SolCarlitz}
  \Sol(\Delta) = \FF_q(t)\cdot f_C(\tpi; t).
\end{equation}
By our analysis in Example~\ref{Ex:omega} we must have $f_{C}(\tpi; t) = \delta \cdot \omega_C(t)$ for some $\delta \in \FF_q(t)$.  Choosing the same root of $-\theta$ as in Example~\ref{Ex:omega} and using \eqref{CarlitzPerFormula}, we find that $f_C(\tpi;t) - \omega_C(t)$ is regular $t=\theta$.  The only element of $\Sol(\Delta)$ that is regular at $t=\theta$ is $0$, and so
\[
  f_{C}(\tpi; t) = \omega_C(t),
\]
which is the same as \cite[Prop.~5.1.3]{ABP04}.

For more general $u \in \CC_\infty$, suppose $\xi = \exp_C(u)$ is within the radius of convergence of $\log_C(z)$, i.e.\ $|\xi| < |\theta|^{q/(q-1)}$.  Then as studied in \cite[\S 6]{P08}, let
\begin{equation} \label{LCdef}
  \cL_C(\xi; t) = \xi + \sum_{n=1}^\infty \frac{\xi^{q^n}}{(t-\theta^q)(t-\theta^{q^2}) \cdots (t-\theta^{q^n})}.
\end{equation}
If $|\xi| = c^{q/(q-1)}$ for $0 \leq c < |\theta|$, then
\[
  \left\lVert \frac{\xi^{q^n}}{(t-\theta^q)(t-\theta^{q^2}) \cdots (t-\theta^{q^n})} \right\rVert
  \leq \left( \frac{c}{|\theta|} \right)^{q^{n+1}/(q-1)} \cdot |\theta|^{q/(q-1)},
\]
which goes to $0$ as $n \to \infty$.  Therefore $\cL_C(\xi; t) \in \TT$, and $\cL_C(\xi; z)$ converges for $|z| < |\theta|^q$.  It follows from \eqref{expClogC} that
\begin{equation} \label{LCtheta}
  \cL_C(\xi;\theta) = \log_C(\xi) = u.
\end{equation}
Now from~\eqref{LCdef} we see that
\[
  \cL_C(\xi; t)^{(1)} = (t-\theta^q) \cL_C(\xi; t)  - (t-\theta^q) \xi,
\]
and so
\[
  \Delta \left(-\frac{\cL_C(\xi;t)}{t-\theta}\right) = \xi.
\]
Subtracting from \eqref{AGFCarlitzdiffeq} and comparing with \eqref{omegaCdiffeq}, we find that
\[
  f_C(u;t) + \frac{\cL_C(\xi;t)}{t-\theta} \in \FF_q(t) \cdot \omega_C(t).
\]
However, by Proposition~\ref{AGFproperties} and \eqref{LCtheta} we see that this expression is regular at $t=\theta$, and so we must have
\begin{equation}
  \cL_C(\xi; t) = -(t-\theta) f_C(u; t)
\end{equation}
whenever $\exp_C(u) = \xi$ and $|\xi| < |\theta|^{q/(q-1)}$.

The function $\cL_C(\xi;t)$ satisfies an additional functional equation in the $\xi$ parameter.  Suppose that both $\theta\xi$ and $\xi^q$ lie within the radius of convergence of $\log_C(z)$.  Then we have
\begin{align*}
  \cL_C(\theta\xi + \xi^q;t) &= \theta\xi + \xi^q + \sum_{n=1}^\infty
  \frac{\theta^{q^n}\xi^{q^n} + \xi^{q^{n+1}}}{(t-\theta^q) \cdots (t-\theta^{q^n})} \\
  &= \theta\xi + \xi^q + \frac{\theta^q\xi^q}{t-\theta^q} + \sum_{n=2}^\infty
  \frac{\theta^{q^n}\xi^{q^n} + \xi^{q^n}}
  {(t-\theta^q) \cdots (t-\theta^{q^{n-1}})} \cdot \frac{t-\theta^{q^n}}{t-\theta^{q^n}} \\
  &= \theta\xi +\xi^q + \frac{\theta^q\xi^q}{t-\theta^q} + t \cL_C(\xi;t) - t\xi - \frac{t\xi^q}{t-\theta^q}\\
  &=t\cL_C(\xi;t) - (t-\theta)\xi.
\end{align*}
Thus we have
\begin{equation} \label{LCcompat}
  \cL_C (C_t(\xi); t) = t \cL_C(\xi; t) - (t-\theta)\xi.
\end{equation}

\begin{remark}
When $\xi \in \oK$, we can specialize \eqref{LCcompat} at $t=\theta$ and recover the identity
\[
  \log_C(\theta \xi + \theta^q) = \theta\log_C(\xi).
\]
However, starting with this equality of Carlitz logarithms, it also follows
from \cite[Thm.~3.1.1]{ABP04} and \cite[Thm.~6.3.2]{P08} that there
must be non-trivial $\oK(t)$-linear relation among $\Omega_C(t)$,
$\Omega_C(t) \cL_C (\xi; t)$, and $\Omega_C(t) \cL_C(C_t(\xi);t)$, that
specializes to produce this identity.  Thus a $\oK(t)$-linear relation
among $1$, $\cL_C(\xi;t)$, and $\cL_C(C_t(\xi);t)$ is expected from
the general theory, and \eqref{LCcompat} provides it explicitly.
\end{remark}

\section{Shadowed partitions} \label{SP}

We now recall information about shadowed partitions and their connections to Drinfeld exponential and logarithm functions from~\cite{EP}.  A partition of a set $S$ is a collection of subsets of $S$ that are pairwise disjoint, and whose union is equal to $S$ itself.  If $S\subseteq \ZZ$ and $j\in \ZZ$, then we set $S+j:=\{i+j: i\in S \}$.  For $r$, $n \in \ZZ^+$, we define
\begin{equation} \label{prndefn}
P_r(n):=\left\{(S_1,\, S_2, \dots,\, S_r) \Biggm| \parbox{2.75in}{
\begin{center} Each $S_i\subseteq \{0,\, 1, \dots,\, n-1\}$, \\
 and $\{S_i+j: 1\leq i \leq r,\, 0\leq j \leq i-1\}$ \\ forms a partition of  $\{0,\, 1, \dots,\, n-1\}$.
 \end{center}} \right\}.
\end{equation}
We also set $P_r(0):=\{(\emptyset, \dots, \emptyset)\}$ to be the singleton set containing the $r$-tuple of empty sets, and $P_r(-n):=\emptyset$.  The elements of $P_r(n)$ are called the order $r$ shadowed partitions of $n$: the sets $S_i$ together with their shadows ($S_i + j$, $1 \leq j \leq i-1$) form a partition of $n$.  We list a few facts about shadowed partitions below, and for more details the reader is directed to \cite{EP}.  For additional examples of using shadowed partitions in the theory of vectorial modular forms and $\tau$-recurrent sequences, see~\cite{EPetrov}.

For $1 \leq i \leq r$, set $P_r^i(n):=\{(S_1,\, S_2, \dots,\, S_r) \in P_r(n): 0 \in S_i\}$.  The sequence of $r$-step Fibonacci numbers $\{F^{[r]}_n\}$ is defined by $F^{[r]}_{-n}=0$, $F^{[r]}_0=1$, $F^{[r]}_n=\sum_{i=n-r}^{n-1} F^{[r]}_i$ ($n \in \ZZ^+$).  Also for $S \subseteq \NN$ finite, we define the weight $w(S)$ of $S$ to be the integer
\[
  w(S) = \sum_{i \in S} q^i.
\]

\begin{lemma}[{See \cite[Lems.~2.1, 2.2]{EP}}] \label{facts1}
Let $n \geq 0$.
\begin{enumerate}
\item[(a)] $\{P^i_r(n): 1 \leq i \leq r \}$ is a partition of $P_r(n)$.
\item[(b)] For $1\leq i \leq r$, there is a bijection $\Pi_i: P_r(n-i) \to P_r^i(n)$ defined by
\[
  \Pi_i : (S_1,\, S_2, \dots,\, S_r)\mapsto (S_1+i,\, S_2+i,\dots,\{0\}\cup(S_i+i),\dots, \, S_r+i).
\]
Thus, $\{ \Pi_i (P_r(n-i)) : 1 \leq i \leq r \}$ is a partition of $P_r(n)$.
\item[(c)] For all $r>0$ we have $|P_r(n)|=F_n^{[r]}$.
\item[(d)] For $1 \leq i \leq r$, there is an injection $\Psi_i : P_r(n-i) \to P_r(n)$ defined by
\[
  \Psi_i : (S_1, \dots, S_r) \mapsto (S_1, \dots, S_i \cup \{ n-i \}, \dots, S_r).
\]
Furthermore, $\{ \Psi_i(P_r(n-i)) : 1 \leq i \leq r \}$ is a partition of $P_r(n)$.
\item[(e)] For $S_i \subseteq \{ 0, 1, \dots, n-1 \}$, the $r$-tuple $(S_1, \dots, S_r)$ is in $P_r(n)$ if and only if
\[
  \sum_{i=1}^r (q^i-1)w(S_i) = q^n-1.
\]
\end{enumerate}
\end{lemma}

\begin{proof}
Parts (a)--(c) and (e) were proved in \cite[Lems.~2.1--2.2]{EP}.  For (d), it suffices by (a) and (b) to show that for $i \neq j$, the images of $\Psi_i$ and $\Psi_j$ are disjoint.  We can assume that $i > j$.  Suppose that $(S_1', \dots, S_r') \in P_r(n)$ and $n-i \in S_i'$.  Then it is not possible for any $k > n-i$ to appear in any of the $S_1', \dots, S_r'$, in particular for $n-j$.  Thus it is not possible for $\Psi_i(S_1, \dots, S_r)$ to be the same as $\Psi_j(\widetilde{S}_1, \dots, \widetilde{S}_r)$ for any choices of shadowed partitions.
\end{proof}

To help simplify our formulas, we shall usually denote $(S_1,\dots, S_r)\in P_r(n)$ by $\bS$. We fix the following notation for the rest of the paper
\[
|{\bS}| :=\sum_{i=1}^r |S_i|,\quad
\bigcup {\bS} :=\bigcup_{i=1}^r S_i.
\]
For $\bA = (A_1,\dots, A_r)\in \CC_\infty^r$ and ${\bS}\in P_r(n)$ we write
\begin{equation}
 \bA^{\bS} := \prod_{i=1}^r A_i^{w(S_i)}.
\end{equation}
Note that ${\bA}^\emptyset = 1$.  We also write
\[
  D_n(\bS) := \prod_{i \in \cup \bS} [n-i]^{q^i}
\]
and
\[
  L(\bS) := \prod_{j=1}^r \prod_{i \in S_j} (-[i+j]).
\]

Now let $\phi$ be a Drinfeld module of rank $r$ over $\CC_\infty$ defined as in \eqref{phidef}, with corresponding exponential and logarithm functions, $\exp_\phi(z)$ and $\log_\phi(z)$, from \eqref{expphidef} and \eqref{logphidef}.  The following theorem demonstrates how to use shadowed partitions to express the power series coefficients of $\exp_\phi(z)$ and $\log_\phi(z)$ concretely in terms the defining coefficients of~$\phi$.

\begin{theorem}[{See \cite[Thms.~3.1, 3.3]{EP}}] \label{alphabeta}
Let $\phi$ be a Drinfeld module of rank $r$, defined by $\phi_t = \theta + A_1 \tau + \dots + A_r \tau^r$, together with exponential $\exp_\phi(z) = \sum \alpha_n z^{q^n}$ and logarithm $\log_\phi(z) = \sum \beta_n z^{q^n}$.  Letting $\bA = (A_1, \dots, A_r)$, we have
\[
  \alpha_n = \sum_{\bS \in P_r(n)} \frac{\bA^{\bS}}{D_n(\bS)}, \quad
  \beta_n = \sum_{\bS \in P_r(n)} \frac{\bA^{\bS}}{L(\bS)}.
\]
\end{theorem}

\begin{example}
We write out a few cases to clarify the formulas in Theorem~\ref{alphabeta}. Let the superscript on $\alpha$ or $\beta$ indicate the rank $r$ of the corresponding module. Then for $r=2$ we obtain
\begin{align*}
\alpha_3^{(2)} &= \frac{A_1^{q^2+q+1}}{[1]^{q^2}[2]^{q}[3]}+\frac{A_1 A_2^q}{[2]^q[3]}
+\frac{A_1^{q^2}A_2}{[1]^{q^2}[3]}, \\
\beta_3^{(2)} &= -\frac{A_1^{q^2+q+1}}{[1][2][3]}+\frac{A_1 A_2^q}{[1][3]}
+\frac{A_1^{q^2}A_2}{[2][3]},
\end{align*}
and for $r=3$,
\begin{align*}
\alpha_3^{(3)} &= \frac{A_1^{q^2+q+1}}{[1]^{q^2}[2]^{q}[3]} + \frac{A_1 A_2^q}{[2]^q[3]} + \frac{A_1^{q^2}A_2}{[1]^{q^2}[3]} + \frac{A_3}{[3]}, \\
\beta_3^{(3)} &= -\frac{A_1^{q^2+q+1}}{[1][2][3]} + \frac{A_1 A_2^q}{[1][3]} +
\frac{A_1^{q^2} A_2}{[2][3]} - \frac{A_3}{[3]}.
\end{align*}
\end{example}

\section{Explicit formulas for Anderson generating functions} \label{ExplAGF}

In this section we prove the main result of this paper, Theorem~\ref{MainThm}, which yields defomations of $\log_\phi(\xi)$ in the Tate algebra $\TT$ and explicit formulas for Anderson generating functions for arbitrary Drinfeld modules.  These functions and formulas depend only on the Drinfeld module's defining coefficients, much like what we saw for the Carlitz module in \S\ref{CarlitzSec}.  Our main tools are shadowed partitions combined with their use in Theorem~\ref{alphabeta}.  We continue with the notation of the previous sections and assume that a Drinfeld module $\phi$ of rank $r$ has been chosen with
\begin{equation}\label{phi}
  \phi_t = \theta + A_1 \tau + \dots + A_r \tau^r, \quad A_r \neq 0.
\end{equation}
When convenient, we will sometimes set $A_0 = \theta$.
Let $\La:=\La_\phi$ be the $\FF_q[\theta]$-lattice of rank $r$ in $\CC_\infty$ that is the kernel of $\exp_\phi(z)$, and let $\{\omega_1,\dots, \omega_r\}$ be an $\FF_q[\theta]$-basis of $\La$.  Fix
\[
  \Delta_\phi = A_r \tau^r + \dots + A_1\tau - (t-\theta) \in \CC_\infty(t)[\tau].
\]
The following result due to Pellarin provides analogues of the results \eqref{AGFCarlitzdiffeq} and \eqref{SolCarlitz} for the Carlitz module for the Anderson generating functions of $\phi$.

\begin{proposition}[{Pellarin~\cite[\S 4.2]{Pellarin08}}] \label{AGFphispace}
For $u \in \CC_\infty$, let
\[
  f_\phi(u; t) = \sum_{n=0}^\infty \frac{\alpha_n u^{q^n}}{\theta^{q^n}-t} \in \TT.
\]
\begin{enumerate}
\item[(a)] For $u \in \CC_\infty$, set $\xi = \exp_\phi(u)$.  Then
\[
  \Delta_\phi \bigl( f_\phi(u; t) \bigr) = \xi.
\]
\item[(b)] The set $\{ f_\phi(\omega_1; t), \dots, f_\phi(\omega_r; t) \} \subseteq \TT$ is an $\FF_q(t)$-basis for $\Sol(\Delta_\phi)$.
\end{enumerate}
\end{proposition}

\begin{proof}
The proof of (a) follows along the similar lines to the arguments for the Carlitz module.  By the functional equation $\phi_t(\exp_\phi(z)) = \exp_\phi(\theta z)$, it follows that for $n \geq 0$,
\[
  \sum_{i=0}^n A_i \alpha_{n-i}^{q^i} = \theta^{q^n} \alpha_n,
\]
where we set $A_0=\theta$, $A_i=0$ if $i > r$, and $\alpha_n = 0$ if $n < 0$.  Therefore,
\begin{align*}
  \Delta_\phi \bigl( f_{\phi}(u; t) \bigr) &= \sum_{i=1}^r A_i \sum_{n=0}^\infty \frac{\alpha_n^{q^i} u^{q^{n+i}}}{\theta^{q^{n+i}} - t} - \sum_{n=0}^\infty \frac{\alpha_n u^{q^n}(t-\theta)}{\theta^{q^n} -t} \\
  &= \sum_{n=0}^\infty \left(\sum_{i=1}^n A_i \alpha_{n-i}^{q^i} \right) \frac{u^{q^n}}{\theta^{q^n}-t}- \sum_{n=0}^\infty \frac{\alpha_n u^{q^n}(t-\theta)}{\theta^{q^n} -t} \\
  &= \sum_{n=0}^\infty \frac{\alpha_n u^{q^n}(\theta^{q^n}-\theta)}{\theta^{q^n}-t}- \sum_{n=0}^\infty \frac{\alpha_n u^{q^n}(t-\theta)}{\theta^{q^n} -t} \\
  &= \sum_{n=0}^\infty \frac{\alpha_n u^{q^n}(\theta^{q^n}-t)}{\theta^{q^n}-t} \\
  &= \xi.
\end{align*}
Now since $\exp_\phi(\omega_j) = 0$ for each $j$, it follows that $f_\phi(\omega_j; t) \in \Sol(\Delta_\phi)$.  From Proposition~\ref{AGFproperties}, we know $\Res_{t=\theta} f_\phi(\omega_j; t) = -\omega_j$.  Thus, any $\FF_q(t)$-linear dependence among $f_\phi(\omega_1; t), \dots, f_\phi(\omega_r; t)$ would induce an $\FF_q(\theta)$-linear dependence of $\omega_1, \dots, \omega_r$, violating their linear independence.  This proves (b).
\end{proof}

\begin{remark}
With a little more work it can be shown in the above proposition that $\{ f_\phi(\omega_1; t), \dots, f_\phi(\omega_r; t) \}$ is also an $\FF_q[t]$-basis for $\Sol(\Delta_\phi) \cap \TT$.  The essence of the argument is to use \cite[Prop.~3.4.7(b)]{P08} in conjunction with the rigid analytic trivialization $\Psi_\phi$ from \cite[Eq.~(3.4.6)]{CP12}.  We omit the details.
\end{remark}

The connection between Anderson generating functions and shadowed partitions has its basis in the following construction.  For $n \in \ZZ$, define the rational function in $t$,
\begin{equation} \label{Bdef}
\cB_n(t):= \sum_{\bS \in P_r(n)}\prod_{i=1}^r \prod_{j \in S_i}\frac{A_i^{q^j}}{t-\theta^{q^{i+j}}}.
\end{equation}
Note that for $n>0$, $\cB_{-n}(t) = 0$, and $\cB_0(t) = 1$.  Moreover, for $n \geq 0$,
a straightforward calculation together with Theorem~\ref{alphabeta} yields
\begin{equation} \label{Bnbetan}
  \cB_n(\theta) = \sum_{\bS \in P_r(n)} \frac{\bA^{\bS}}{L(\bS)} = \beta_n.
\end{equation}
It is thus natural to consider the series
\begin{equation}\label{Ldef}
 \cL_\phi(\xi;t) := \sum_{n=0}^\infty \cB_n(t) \xi^{q^n},
\end{equation}
as a deformation of $\log_\phi(\xi)$.

\begin{lemma}\label{DegSummands}
Let $\phi$ be a Drinfeld module of rank $r$ given by \eqref{phi}. Set
\[
N(\phi) := \{1\leq i\leq r: A_i\neq 0\},
\]
and
\[
P_r(n;\phi)=\{\bS=(S_1,S_2, \dots, S_r)\in P_r(n): S_i=\emptyset \textnormal{\ for all\ } i\notin N(\phi)\}.
\]
For $\bS \in P_r(n)$, set
\[
X_\phi(\bS;t):=\prod_{i=1}^r \prod_{j \in S_i}\frac{A_i^{q^j}}{t-\theta^{q^{i+j}}}.
\]
\begin{enumerate}
\item[(a)] We have $X_\phi(\bS;t)=0$ if and only if $\bS \notin P_r(n;\phi)$, whence
\[
\cB_n(t)=\sum_{\bS \in P_r(n;\phi)}X_\phi(\bS;t).
\]
\item[(b)]
Define rational numbers $\mu_{ik}$ for $i$, $k \in N(\phi)$ by
\[
\mu_{ik}:=\frac{\deg(A_k)-q^k}{q^k-1}-\frac{\deg(A_i)-q^i}{q^i-1}.
\]
Then for all $\bS \in P_r(n;\phi)$ and any $i\in N(\phi)$ we have
\[
\log_q\lVert X_\phi(\bS;t)\rVert=\frac{q^n-1}{q^i-1}(\deg(A_i)-q^i)+\sum_{k\in N(\phi)}(q^k-1)w(S_k)\mu_{ik},
\]
and also, for all $z\in \CC_\infty$ with $|z|< |\theta|^q$ we have
\[
\deg (X_\phi(\bS;z))=\frac{q^n-1}{q^i-1}(\deg(A_i)-q^i)+\sum_{k\in N(\phi)}(q^k-1)w(S_k)\mu_{ik}.
\]
\end{enumerate}
\end{lemma}

\begin{proof}
Part (a) is immediate from the definitions of $\cB_n(t)$ and $X_\phi(\bS;t)$. To prove (b) we apply the geometric series formula to $X_\phi(\bS;t)$ for $\bS\in P_r(n;\phi)$ to obtain
\[
X_\phi(\bS;t)=-\sum_{m=0}^\infty \left(\prod_{i\in N(\phi)} \prod_{j \in S_i}\frac{A_i^{q^j}}{\theta^{(m+1)q^{i+j}}}\right) t^m=-\sum_{m=0}^\infty \left(\prod_{i\in N(\phi)} \left(\frac{A_i}{\theta^{(m+1)q^i}}\right)^{w(S_i)}\right) t^m,
\]
from which it is clear that the maximal $\lvert\, \cdot\,\rvert$ value for a $t^m$ coefficient is achieved at $m=0$. We thus have
\[
\log_q\lVert X_\phi(\bS;t)\rVert=\sum_{i\in N(\phi)} w(S_i)\big(\deg(A_i)-q^i\big),
\]
and a straightforward computation using  Lemma~\ref{facts1}(e) delivers the first formula in (b). To prove the second, note that since $i\geq 1$ and $\deg z<q$, we have $\deg(z-\theta^{q^{i+j}})=q^{i+j}$ for any $j\in S_i$. We thus find
\[
\deg(X_\phi(\bS;z))=\sum_{i\in N(\phi)} w(S_i)\big(\deg(A_i)-q^i\big),
\]
and the formula follows exactly as above.
\end{proof}

To understand the behavior of $\lVert\cB_n(t)\rVert$ and $\deg(\cB_n(z))$, we need to identify the maximal values for $\lVert X_\phi(\bS;t)\rVert$ and $\deg(X_\phi(\bS;z))$ across all $\bS\in P_r(n;\phi)$. For the remainder of the paper we will let $s=s(\phi)$ denote the (smallest, if there is more than one) index $s\in N(\phi)$ such that
\begin{equation}\label{sdef}
\frac{\deg(A_s)-q^s}{q^s-1}\geq \frac{\deg(A_i)-q^i}{q^i-1}, \quad\textrm{ for all } i\in N(\phi).
\end{equation}

\begin{corollary}\label{upperbound}
Let $\phi$ be a Drinfeld module as in \eqref{phi}, and let $s=s(\phi)$ be as in \eqref{sdef}. Then  for all $n\geq 0$,
\[
\log_q\lVert \cB_n(t)\rVert \leq \frac{q^n-1}{q^s-1} (\deg(A_s)-q^s).
\]
Also, for $z\in \CC_\infty$ with $|z|< |\theta|^q$ we have
\[
\deg( \cB_n(z)) \leq \frac{q^n-1}{q^s-1} (\deg(A_s)-q^s).
\]
\end{corollary}

\begin{proof}
Our choice of $s$ implies that $\mu_{sk}\leq 0$ for all $k \in N(\phi)$.  Both formulas follow at once from Lemma~\ref{DegSummands}(b), combined with \eqref{Bdef} and the ultrametric inequality.
\end{proof}

It could happen that the maximal value of $\lVert X_\phi({\bf S};t)\rVert$ is achieved for two distinct shadowed partitions $\bS_1\neq \bS_2$; in which case $\lVert \cB_n(t)\rVert$ might be, for some values of $n$, smaller than the bound in Corollary~\ref{upperbound}. Nonetheless, the corollary enables us to discuss the convergence of $\cL_\phi(\xi;t)$ in $\TT$ for $|\xi|$ sufficiently small, as we discuss in the next proposition.

\begin{proposition}\label{rocprop}
Let $\phi$ be a Drinfeld module as in \eqref{phi}, and set
\[
R_\phi := |\theta|^{(q^s-\deg(A_s))/(q^s-1)},
\]
where $s=s(\phi)$ is chosen as in \eqref{sdef}.
\begin{enumerate}
\item[(a)]
If $\xi\in \CC_\infty$ is such that $|{\xi}|<R_\phi$, then the series \eqref{Ldef} defining $\cL_\phi(\xi;t)$ converges in $\TT$ with respect to the $\lVert\, \cdot\, \rVert$-norm.  In this case, then also for all $z\in \CC_\infty$ with $|z| < |\theta|^q$, the series $\cL_\phi(\xi;z)$ converges in $\CC_\infty$.
\medskip
\item[(b)]
If we have that the inequality in \eqref{sdef} is strict when $i \neq s$, then the series definining $\cL_\phi(\xi;t)$ does not converge in $\TT$ with respect to the $\lVert\,\cdot\, \rVert$-norm for any $|\xi|\geq R_\phi$.  In this case, then also for any fixed $z \in \CC_\infty$ with $|z| < |\theta|^q$,  $R_\phi$ is the exact radius of convergence in $\xi$ of $\cL_\phi(\xi;z)$.
\end{enumerate}
\end{proposition}

\begin{proof}
The convergence statements follow directly from Corollary \ref{upperbound}.  When the inequality in \eqref{sdef} is strict, we see from Lemma~\ref{DegSummands}(b) with $i=s$ that
\[
\log_q \lVert X_\phi(\bS;t)\rVert=\frac{q^n-1}{q^s-1}(\deg(A_s)-q^s)+\sum_{k\in N(\phi),\, k\neq s}(q^k-1)w(S_k)\mu_{sk}=\deg(X_\phi(\bS;z)),
\]
with $\mu_{sk}<0$ for all $k\in N(\phi)\setminus \{s\}$. Thus if $s|n$, then the maximal value is uniquely achieved for $\bS$ in which $S_i=\emptyset$ for $i\neq s$ and $S_s=\{0, s, 2s, \dots, n-s\}$. For such $n$ we have
\[
\log_q\lVert \cB_n(t)\xi^{q^n}\rVert=\frac{q^n}{q^s-1}\big(\deg(A_s)-q^s+(q^s-1)\deg(\xi)\big)+\frac{q^s}{q^s-1}=\deg\left(\cB_n(z)\xi^{q^n}\right),
\]
the limit of which is $-\infty$ if and only if $\deg(\xi)<\frac{q^s-\deg(A_s)}{q^s-1}$, and part (b) follows.
\end{proof}

\begin{remark}
Under the conditions of Proposition~\ref{rocprop}(b), it will follow from Theorem~\ref{MainThm}(b) that $R_\phi$ is the radius of convergence in $\xi$ of $\log_\phi(\xi)$.
Even if the inequality in \eqref{sdef} is not strict for some $i \neq s$, we speculate that the conclusion of Proposition~\ref{rocprop}(b) still holds.  Indeed the authors have shown \cite[Lem.~4.1, Cor.~4.2]{EP} that this is the case when the rank of $\phi$ is $2$.  However, in higher ranks there are many middle cases to consider that require more involved combinatorial analysis.
\end{remark}

We also have the following recursive formulas for $\cB_n(t)$.

\begin{lemma}\label{Blemma}
For $m \geq 1$, the sequence $\cB_m(t)$ satisfies the following recurrences:
\begin{enumerate}
\item[(a)] $\displaystyle \cB_m(t)=\sum_{k=1}^r \frac{A_k}{t-\theta^{q^k}} \cdot \cB_{m-k}(t)^{(k)}$,
\item[(b)] $\displaystyle \cB_m(t) =\sum_{k=1}^r \frac{A^{q^{m-k}}_k}{t-\theta^{q^m}}\cdot \cB_{m-k}(t)$.
\end{enumerate}
\end{lemma}

\begin{proof}
{From} the definition of $\cB_n(t)$ we see that
\begin{align*}
\frac{A_k}{t-\theta^{q^k}} \cdot \cB_{m-k}(t)^{(k)}
&=\frac{A_k}{t-\theta^{q^k}} \sum_{\bS \in P_r(m-k)}\prod_{i=1}^r \prod_{j\in S_i} \frac{A_i^{q^{j+k}}}{(t-\theta^{q^{i+j+k}})}\\
&=\sum_{\bS \in P^k_r(m)}\prod_{i=1}^r \prod_{j\in S_i} \frac{A_i^{q^j}}{(t-\theta^{q^{i+j}})},
\end{align*}
where we have used the fact that the map
\begin{gather*}
\Pi_k : P_r(m-k)\to P^k_r(m)\\
(S_1,\dots, S_k, \dots, S_r)\mapsto (S_1+k, \dots, (S_k+k)\cup\{0\}, \dots, S_r+k)
\end{gather*}
is a bijection from Lemma~\ref{facts1}(b). Part (a) follows since the sets $P_r^k(m)$, $1 \leq k \leq r$, form a partition of $P_r(m)$.
To prove (b) note that
\[
 \frac{A_k^{q^{m-k}}}{t-\theta^{q^m}}\sum_{{\bf S}\in P_r(m-k)}\prod_{i=1}^r \prod_{j\in S_i} \frac{A_i^{q^j}}{t-\theta^{q^{i+j}}}= \sum_{{\bf S} \in \Psi_k(P_r(m-k))}\prod_{i=1}^r \prod_{j\in S_i} \frac{A_i^{q^j}}{t-\theta^{q^{i+j}}}.
\]
Since the sets $\Psi_k(P_r(m-k))$, $1 \leq k \leq r$, form a partition of $P_r(m)$ by Lemma~\ref{facts1}(d), we are done.
\end{proof}

We now state and prove the main theorem (Theorem~\ref{MainThm}) of the present paper. 
The theorem shows that $\cL_\phi(\xi;t)$ is a deformation of $\log_\phi(\xi)$ in the $t$ parameter and that it plays the role of the function $\cL_C(\xi;t)$ from \eqref{LCdef}.  We also see that $\cL_\phi(\xi;t)$ is explicitly linked to the Anderson generating function $f_\phi(u;t)$, just as we have for the case of the Carlitz module in \S\ref{CarlitzSec}.

\begin{theorem} \label{MainThm}
Let $\phi : \FF_q[t] \to \CC[\tau]$ be a Drinfeld module of rank $r$ defined by
\[
\phi_t = \theta + A_1\tau + \dots + A_r\tau^r.
\]
For $u \in \CC_\infty$ and $\xi = \exp_\phi(u)$, let
\[
  \cL_\phi(\xi;t) = \sum_{n=0}^\infty \cB_n(t) \xi^{q^n} = \xi
  + \sum_{n=1}^\infty \left( \sum_{\bS \in P_r(n)} \prod_{i=1}^r \prod_{j\in S_i}
  \frac{A_i^{q^j}}{t-\theta^{q^{i+j}}} \right) \xi^{q^n}.
\]
Suppose that $|\xi| < R_\phi$, where $R_\phi$ is defined in Proposition~\ref{rocprop}.
Then $\cL_\phi(\xi;t)$ satisfies the following properties.
\begin{enumerate}
\item[(a)] \textup{(Convergence):} The defining series for $\cL_\phi(\xi; t)$ converges in $\TT$, and for $z \in \CC_\infty$ with $|z| < |\theta|^q$, $\cL_\phi(\xi;z)$ converges in $\CC_\infty$.
\item[(b)] \textup{(Specialization):} We have
\[
  \cL_\phi(\xi;\theta) = \log_\phi(\xi) = u.
\]
\item[(c)] \textup{($\tau$-Difference relation):} Let $\Delta_\phi = A_r \tau^r + \dots + A_1\tau - (t-\theta)$.  Then
\[
  \Delta_\phi\left( -\frac{\cL_\phi(\xi;t)}{t-\theta} \right) = \xi.
\]
\item[(d)] \textup{(Link with Anderson generating functions):} As elements of $\TT$,
\[
  \cL_\phi(\xi;t) = -(t-\theta) f_\phi(u;t).
\]
\item[(e)] \textup{(Compatibility with $\phi$):} If $|A_i \xi^{q^i}| < R_\phi$ for $0 \leq i \leq r$, then
\[
  \cL_\phi \bigl( \phi_{t}(\xi); t) = t \cL_\phi(\xi; t) - (t-\theta) \xi.
\]
\end{enumerate}
\end{theorem}

\begin{proof}
The convergence of $\cL_\phi(\xi; t)$ in $\TT$ and $\cL_\phi(\xi;z)$ in $\CC_\infty$ follow from Proposition~\ref{rocprop}, thus proving~(a).  Now as $\cB_n(\theta) = \beta_n$ by \eqref{Bnbetan}, taking $t = z = \theta$ it follows that
\begin{equation} \label{Lphispecialize}
  \cL_\phi(\xi; \theta) = \sum_{n=0}^\infty \beta_n \xi^{q^n} = \log_\phi(\xi) = u,
\end{equation}
which proves part~(b).  For~(c) we perform the calculation,
\begin{align*}
  \Delta_\phi \left( - \frac{\cL_\phi(\xi; t)}{t-\theta} \right) &=
  -\sum_{k=1}^r A_k \left( \frac{\cL_\phi(\xi;t)}{t-\theta}\right)^{(k)} + \cL_\phi(\xi;t)\\
  &= -\sum_{k=1}^r \sum_{n=0}^\infty \frac{A_k}{t-\theta^{q^k}} \cdot \cB_n(t)^{(k)}\xi^{q^{n+k}} + \sum_{n=0}^\infty \cB_n(t) \xi^{q^n} \\
  &= \sum_{m=0}^\infty \left( \cB_m(t) - \sum_{k=1}^r \frac{A_k}{t - \theta^{q^k}} \cdot \cB_{m-k}(t)^{(k)}
  \right) \xi^{q^m} \\
  &= \xi,
\end{align*}
where the final equality follows from Lemma~\ref{Blemma}(a).  Now for (d), we see from Proposition~\ref{AGFphispace} that
\[
  \frac{\cL_{\phi}(\xi;t)}{t-\theta} + f_\phi(u;t) \in \Sol(\Delta_\phi).
\]
However, by Proposition~\ref{AGFproperties} and \eqref{Lphispecialize} it follows that this expression is regular at $t=\theta$.  The only element of $\Sol(\Delta_\phi)$ that is regular at $t=\theta$ is $0$ by Propositions~\ref{AGFproperties}(b) and~\ref{AGFphispace}(b), which proves (d).  Finally for (e), we argue as in \eqref{LCcompat}: as elements of $\TT$,
\begin{align*}
  \cL_\phi \left( \theta \xi + \sum_{k=1}^r A_k \xi^{q^k}; t \right) &=
  \theta \xi + \sum_{k=1}^r A_k \xi^{q^k} + \sum_{n=1}^\infty \cB_n(t)
  \left( \theta^{q^n}\xi^{q^n} + \sum_{k=1}^r A_k^{q^n}\xi^{q^{k+n}} \right) \\
  &= \theta\xi + \sum_{m=1}^\infty \left( \theta^{q^m} \cB_m(t)+
  \sum_{k=1}^r A_k^{q^{m-k}} \cB_{m-k}(t) \right) \xi^{q^m}\\
  &= \begin{aligned}[t]
    \theta\xi
    + \sum_{m=1}^\infty{}& \Biggl( t\cB_m(t) \xi^{q^m}
    - (t-\theta^{q^m}) \\
    &{}\times\Biggl( \cB_m(t) -\sum_{k=1}^r \frac{A_k^{q^m}}{t-\theta^{q^m}} \cdot \cB_{m-k}(t)
    \Biggr) \Biggr) \xi^{q^m}
    \end{aligned} \\
  &= t \cL_\phi(\xi;t) -(t-\theta)\xi.
\end{align*}
We note that the reordering of the terms in the second equality is justified by the chosen bounds on $|A_i\xi^{q^i}|$.  Also, the disappearance of the $\sum_{k=1}^r A_k \xi^{q^k}$ term in the second equality occurs when reordering the sum.  The last equality follows from Lemma~\ref{Blemma}(b).
\end{proof}

\section{Applications to quasi-periods} \label{ApplQP}

In the theory of elliptic curves over $\CC$, an extension of an elliptic curve by $\Ga$ is uniformized using the Weierstrass $\zeta$-function, and the coordinates of the periods of the extension are called quasi-periods and are obtained through taking values of $\zeta$ at half-periods.  Anderson, Deligne, Gekeler, and Yu developed a rich analogy in the context of Drinfeld modules (see \cite{Gekeler89a}, \cite{Gekeler89b}, \cite{Yu90}), where one studies isomorphism classes of extensions of a Drinfeld module by $\Ga$.  Subsequently extensions of higher dimensional $t$-modules by $\Ga$ were studied by Brownawell and the second author~\cite{BP02}, together with their quasi-periodic functions and quasi-periods.

For a Drinfeld module $\phi$ of rank $r$ over $\CC_\infty$, one can choose $r-1$ quasi-periodic functions
\[
  F_{\phi,j}(z) := \sum_{n=1}^\infty b_{j,n} z^{q^n}, \quad 1 \leq j\leq r-1,
\]
which are entire and $\FF_q$-linear, such that
\[
  F_{\phi,j}(\theta z) - \theta F_{\phi,j}(z) = \exp_\phi(z)^{q^j}.
\]
In this way the $r$-tuple of functions $(\exp_\phi(z_0), z_0 + F_{\phi,1}(z_1), \dots, z_0 + F_{\phi,r-1}(z_{r-1}))$, where $z_0, \dots, z_{r-1}$ are independent variables, provide the exponential map on a maximal quasi-periodic extension of $\phi$ by $\Ga^{r-1}$.  Moreover, if $\omega \in \La_\phi$ is a period of $\phi$, then the $r$-tuple $(\omega, -F_{\phi,1}(\omega), \dots, {-F_{\phi,r-1}(\omega)})$ is a period of this quasi-periodic extension, and so
\[
  F_{\phi,1}(\omega), \dots,  F_{\phi,r-1}(\omega),
\]
are called the quasi-periods associated to $\omega$.  If $\{\omega_1, \dots, \omega_r\}$ is an $\FF_q[\theta]$-basis of $\Lambda_\phi$, then the $r\times r$ matrix
\[
  \rP = \begin{pmatrix}
  \omega_1 & F_{\phi,1}(\omega_1) & \cdots & F_{\phi,r-1}(\omega_1) \\
   \vdots & \vdots & & \vdots \\
  \omega_r & F_{\phi,1}(\omega_r) & \cdots & F_{\phi,r-1}(\omega_r)
  \end{pmatrix}
\]
forms the period matrix of $\phi$, which provides the isomorphism between the de Rham and Betti modules associated to $\phi$.  The reader is directed to \cite{Gekeler89a}, \cite{Pellarin08}, \cite{Yu90} for more details on these constructions.

As noted by Gekeler~\cite[\S 2]{Gekeler89a}, one can also capture quasi-periods through Anderson generating functions (see also Pellarin~\cite[\S 4]{Pellarin08}).  Moreover, if $\omega \in \La_\phi$ and $f_\phi(\omega; t)$ is the Anderson generating function associated to $\omega$, then for $1 \leq j \leq r-1$,
\begin{equation} \label{QPformula}
  F_{\phi,j}(\omega) = f_{\phi}(\omega,t)^{(j)}\Bigr|_{t=\theta} = \sum_{m = 0}^\infty
  \exp_\phi \left( \frac{\omega}{\theta^{m+1}} \right)^{q^j} \theta^j.
\end{equation}

We can then use Theorem~\ref{MainThm} to find formulas for quasi-periods in terms of $\cL_\phi(\xi;t)$, which require calculating only finitely many $t$-power division points on $\phi$.  Indeed, for $u \in \CC_\infty$,
\[
  \lim_{\ell \to \infty} \exp_\phi\left( \frac{u}{\theta^\ell} \right) = 0,
\]
so we can pick $\ell \geq 0$ so that
\[
  \xi := \exp_\phi \left( \frac{u}{\theta^\ell} \right) \quad \Rightarrow \quad |\xi| < R_\phi.
\]
By \eqref{AGFdef},
\[
  f_\phi(u;t) = t^\ell f_\phi \left( \frac{u}{\theta^\ell}; t \right) + \sum_{m=0}^{\ell-1}
  \exp_\phi \left( \frac{u}{\theta^{m+1}} \right) t^m,
\]
and therefore by Theorem~\ref{MainThm}(d),
\begin{equation}
  f_\phi(u;t) = -\frac{t^{\ell}}{t-\theta} \cdot \cL_\phi ( \xi ; t )
  + \sum_{m=0}^{\ell-1} \exp_\phi \left( \frac{u}{\theta^{m+1}} \right) t^m.
\end{equation}
We then have the following proposition that follows immediately from Proposition~\ref{AGFproperties} and \eqref{QPformula}.

\begin{proposition} \label{PerQP}
For $\omega \in \Lambda_\phi$, pick $\ell \geq 0$ so that $\zeta := \exp_\phi(\omega/\theta^{\ell})$ satisfies $|\zeta| < R_\phi$.  Then
\[
  \omega = \theta^{\ell} \cL_\phi(\zeta;\theta),
\]
and for $1 \leq j \leq r-1$,
\[
  F_{\phi,j}(\omega) = \frac{\theta^{\ell}}{\theta^{q^j} - \theta} \cdot \cL_\phi(\zeta; t)^{(j)}\Bigr|_{t=\theta} + \sum_{m=0}^{\ell-1} \exp_\phi \left( \frac{\omega}{\theta^{m+1}} \right)^{q^j} \theta^m.
\]
\end{proposition}

\begin{example}[Anderson's Legendre Relation]
We conclude by considering the Legendre relation for rank~$2$ Drinfeld modules, which was proved by Anderson (e.g., see \cite[\S 2.5]{CP11}, \cite[\S 4.2]{Pellarin08}, or \cite[\S 2.1]{Pellarin13} for discussions about it).
Let $\phi$ be a rank $2$ Drinfeld module defined by
\[
  \phi_t = \theta + A \tau + B\tau^2,
\]
such that
\[
  j(\phi) := \frac{A^{q+1}}{B} \quad\textup{and}\quad \deg j(\phi) < q^2.
\]
This choice is not essential for the Legendre relation, which holds for arbitrary $\phi$ (even of higher rank), but it simplifies our considerations somewhat.  It follows from \cite[Thm.~5.3]{EP} that each $t$-torsion point $\zeta$ on $\phi$ satisfies $|\zeta| < R_\phi$ (so $\ell=1$).  Therefore, if we pick generators $\omega_1$, $\omega_2$ for $\Lambda_\phi$ and set
\[
  \zeta_i := \exp_\phi \left( \frac{\omega_i}{\theta} \right), \quad i=1, 2,
\]
then we can consider
\[
  \cP(t) := \begin{pmatrix}
  \dfrac{-t}{t-\theta} \cdot \cL_\phi(\zeta_1; t) + \zeta_1 &
  \dfrac{-t}{t-\theta^q} \cdot \cL_\phi(\zeta_1; t)^{(1)} + \zeta_1^q \\[10pt]
  \dfrac{-t}{t-\theta} \cdot \cL_\phi(\zeta_2; t) + \zeta_2 &
  \dfrac{-t}{t-\theta^q} \cdot \cL_\phi(\zeta_2; t)^{(1)} + \zeta_2^q
  \end{pmatrix}.
\]
The entries of $\cP(t)$ are in $\TT$.  By Proposition~\ref{PerQP}, the residues at $t=\theta$ of entries of the first column of $\cP(t)$ are $-\omega_1$ and $-\omega_2$ respectively, and the entries of the second column evaluated at $t=\theta$ are the quasi-periods $\eta_1 := F_{\phi,1}(\omega_1)$ and $\eta_2 := F_{\phi,1}(\omega_2)$.  A relatively straightforward calculation using Theorem~\ref{MainThm}(c) yields
\[
  \det \cP(t)^{(1)} = -\left( \frac{t-\theta}{B} \right) \det \cP(t).
\]
It follows from the discussion in Example~\ref{Ex:omega} that by choosing $(-B)^{1/(q-1)}$ appropriately,
\[
  \det \cP(t) = \frac{\omega_C(t)}{(-B)^{1/(q-1)}}.
\]
Therefore,
\begin{equation}
  \omega_1 \eta_2 - \omega_2\eta_1 = -\Res_{t=\theta} \bigl( \det \cP(t)\bigr) = \frac{\tpi}{(-B)^{1/(q-1)}},
\end{equation}
which provides an analogue of the classical Legendre relations for elliptic curves.  At the expense of more complicated formulas, one can also use the present techniques to investigate the Legendre relation for arbitrary Drinfeld modules (cf.~\cite[\S 3.4]{CP12}, \cite[\S 4.2]{Pellarin08}).
\end{example}

\end{document}